\documentclass[11pt]{article}
\usepackage[a4paper,margin=25mm]{geometry}
\usepackage{amsmath,amssymb,amsthm,mathtools}
\usepackage{capt-of}
\usepackage{microtype}
\usepackage{xcolor}
\usepackage{aliascnt}
\usepackage{hyperref}
\hypersetup{colorlinks=true,linkcolor=cyan,citecolor=red,urlcolor=black,filecolor=black}
\DeclareMathOperator{\Comp}{Comp}
\allowdisplaybreaks[2]
\numberwithin{equation}{section}
\newtheorem{theorem}{Theorem}[section]
\newaliascnt{proposition}{theorem}
\newtheorem{proposition}[proposition]{Proposition}
\aliascntresetthe{proposition}
\newaliascnt{lemma}{theorem}
\newtheorem{lemma}[lemma]{Lemma}
\aliascntresetthe{lemma}
\newaliascnt{corollary}{theorem}
\newtheorem{corollary}[corollary]{Corollary}
\aliascntresetthe{corollary}
\newtheorem*{conjecture}{Conjecture}
\theoremstyle{remark}

\begin{document}
\begin{center}
{\Large\bf
Negative edge correlation in complete bipartite graphs
}
\end{center}

\begin{center}
Ying Cao$^1$\quad\quad Houshan Fu$^{2,*}$\par\vspace{.2cm}

$^1$Department of Mathematics, Hong Kong University of Science and Technology\\
Clear Water Bay, Hong Kong\\\par
\vspace{.2cm}
$^2$School of Mathematics and Information Science, Guangzhou University\\
Guangzhou 510006, Guangdong, P. R. China\par
\vspace{.2cm}
$^*$Correspondence to be sent to: fuhoushan@gzhu.edu.cn \\
Email: ycaobf@connect.ust.hk
\end{center}

\begin{abstract}
{We prove strict pairwise negative correlation under both the
uniform forest measure and the uniform connected-subgraph measure on the
complete bipartite graph $K_{m,n}$ for all $m,n\geq2$.
Under the connected-subgraph measure, any two distinct edges of a fixed
copy of $K_{2,2}$ remain strictly negatively correlated after conditioning
on the subgraph induced by the remaining vertices.}

\end{abstract}

\section{Introduction}
All graphs are finite and, unless stated otherwise, simple and undirected.
We write $V(G)$ and $E(G)$ for the vertex and edge sets of
a graph $G$.
Let $T$ be a uniformly chosen spanning tree of a connected
graph $G$. For distinct edges $e,f\in E(G)$,
\[
 \mathbb P(e,f\in E(T))
 \leq\mathbb P(e\in E(T))\mathbb P(f\in E(T)).
\]
Thus, the edge indicators are pairwise negatively correlated; see
\cite[Section~1]{GrimmettWinkler2004}. A natural question is whether the same inequality holds when the spanning
tree is replaced by a uniformly chosen forest or connected spanning
subgraph. 

For a graph $G$, define
\[
 \begin{aligned}
 \mathcal F(G)&\coloneqq\{(V(G),X):X\subseteq E(G),\ (V(G),X)\text{ is acyclic}\},\\
 \mathcal C(G)&\coloneqq\{(V(G),X):X\subseteq E(G),\ (V(G),X)\text{ is connected}\}.
 \end{aligned}
\]
A forest in $ \mathcal F(G)$ need not be connected or maximal. In particular,
$(V(G),\varnothing)\in\mathcal F(G)$.
Choosing uniformly from $\mathcal F(G)$ gives the \emph{uniform forest
measure}. For connected $G$, the uniform measure on $\mathcal C(G)$ is
the \emph{uniform connected-subgraph measure}.
For $X\subseteq E(G)$, let $\mathcal F_X(G)$ and
$\mathcal C_X(G)$ denote the respective subfamilies consisting of subgraphs
that contain every edge of $X$.
Following Oxley \cite[Section~15.8, p.~603]{Oxley2011}, define
\[
i_X(G)\coloneqq|\mathcal F_X(G)|,\qquad j_X(G)\coloneqq|\mathcal C_X(G)|.
\]
We omit empty subscripts and subscript braces. Thus,
$i(G)=i_\varnothing(G)$, $i_e(G)=i_{\{e\}}(G)$, and
$i_{ef}(G)=i_{\{e,f\}}(G)$. The forest and connected-subgraph
conjectures take the following form.
\begin{conjecture}[\cite{GrimmettWinkler2004}, Conjecture~1.1]
For every finite graph $G$ and distinct edges $e,f\in E(G)$,
\[
 i_e(G)i_f(G)-i_{ef}(G)i(G)\geq0,\qquad  j_e(G)j_f(G)-j_{ef}(G)j(G)\geq0.
\]
\end{conjecture}
These inequalities express pairwise negative correlation under
the uniform forest measure and, for connected $G$, under the uniform
connected-subgraph measure. This property is called
\emph{edge-negative-association} in \cite{GrimmettWinkler2004}.
For the forest conjecture, Pemantle
\cite[Section~3.1]{Pemantle2000} records a question raised by Peter
Winkler. Oxley~\cite[p.~603]{Oxley2011} attributes this conjecture to
Kahn~\cite{Kahn2000} and the connected-subgraph conjecture to Grimmett
and S.~N.~Winkler.

We study both conjectures on complete bipartite graphs, allowing their two
part sizes to vary independently.
Throughout, let $L$ and $R$ be disjoint sets with $|L|=m$ and $|R|=n$.
The complete bipartite graph $K_{m,n}$ has vertex set $L\cup R$ and edge set
\[
 E(K_{m,n})=\bigl\{\{u,v\}:u\in L,\ v\in R\bigr\}.
\]
Unless stated otherwise, the part sizes satisfy $m,n\geq2$.
Our first theorem proves the forest conjecture for complete bipartite graphs.

\begin{theorem}\label{thm:main-forest}
For integers $m,n\geq2$ and distinct edges $e,f$ of $K_{m,n}$,
\[
 i_e(K_{m,n})i_f(K_{m,n})-i_{ef}(K_{m,n})i(K_{m,n})>0.
\]
\end{theorem}

Our second theorem gives the corresponding result for connected
spanning subgraphs.

\begin{theorem}\label{thm:main-connected}
For integers $m,n\geq2$ and distinct edges $e,f$ of $K_{m,n}$,
\[
 j_e(K_{m,n})j_f(K_{m,n})-j_{ef}(K_{m,n})j(K_{m,n})>0.
\]
\end{theorem}

If either part has size one, every spanning subgraph is a forest, so
the edge indicators are independent under the uniform forest measure.
The only connected spanning subgraph is $K_{m,n}$ itself.
Hence, both correlation differences are zero.

Negative correlation for uniform matroid bases was studied by Seymour
and Welsh \cite{SeymourWelsh1975}, and Feder and Mihail
\cite{FederMihail1992} introduced balanced matroids.
For uniform forests, Grimmett and Winkler
\cite[Theorem~1.3]{GrimmettWinkler2004} verified the conjecture
for graphs with at most eight vertices and for nine-vertex graphs with
at most eighteen edges. For arbitrary graphs, Br\"and\'en and Huh
\cite[Propositions~4.21 and~4.25]{BrandenHuh2020} proved
$i_{ef}(G)i(G)\leq2i_e(G)i_f(G)$. However, the conjecture requires the factor $1$.

For complete graphs, Stark \cite{Stark2011} proved the forest conjecture
for sufficiently large orders, and Gupta
\cite[Theorem~1.2]{Gupta2026} established it for every $K_n$ with $n\geq2$.
For complete bipartite graphs, Stark \cite[Theorem~1]{Stark2013}
obtained an asymptotic formula for $i(K_{m,n})$ as $m\to\infty$,
with $m\leq n$ and $n=o(m^{6/5})$.
In an earlier preprint, Stark \cite[Theorem~3]{StarkBipartite2011}
proved strict negative correlation for adjacent edges as $m\to\infty$,
with $m\leq n$ and $n=o(m^{8/7})$, leaving the disjoint-edge case unresolved.
For enumeration, Stones \cite{Stones2014} gives an algorithm for counting
forests of $K_{m,n}$ with a prescribed number of edges.
Other forest results concern adjacent edges at sufficiently large common edge
weights on each fixed finite connected graph
\cite[Theorem~1]{Huang2024} and strict negative correlation for uniform
forests on wheel graphs \cite[Theorem~1.1]{BadoWheel2026}.

For connected spanning subgraphs, Tang and Zhang
\cite[Part~I, Theorem~1.4]{TangZhang2026} established the uniform
inequality for sufficiently large complete graphs.
\c{C}i\c{c}eksiz and Ravichandran
\cite[Corollary~1.10]{CiceksizRavichandran2026}
proved the uniform inequality for complete graphs of every order.
The random-cluster result of Tang and Zhang for high-degree regular graphs
\cite[Theorem~1.2, Corollary~6.2 and Section~6.3]{TangZhangRegular2026}
includes the uniform connected-subgraph measure on $K_{n,n}$ for
sufficiently large $n$.
This regular-graph result does not cover $K_{m,n}$ with $m\ne n$.

Weighted versions are known for several graph classes. Semple and Welsh
\cite[Proposition~3.7 and Theorem~4.4]{SempleWelsh2008}
proved the forest and connected-subgraph inequalities for graphs obtained
by series and parallel extensions from graphs whose maximal $2$-connected subgraphs are $K_4$
or its minors; this class includes $K_{2,n}$.
Wagner \cite[Theorem~5.8(b),(c) and Example~5.2]{Wagner2008} established
preservation under two-sums and reported a computational verification
of the weighted forest inequality for $K_{3,4}$.
Erickson \cite[Theorem~1.4]{Erickson2012} expressed the unnormalized forest
correlation difference for series-parallel graphs as a sum of monomials times squares.

For $K_{3,3}$, Erickson's calculation for the cographic matroid
\cite[Section~5]{Erickson2012}, together with
\cite[Proposition~3.3]{SempleWelsh2008}, implies negative correlation
for the connected-subgraph measure under arbitrary positive edge weights.
At unit weights, Erickson's basis lower bound equals the unnormalized
spanning-tree correlation difference for $K_{3,3}$, which is positive
for every edge pair.
Thus, the uniform connected-subgraph inequality is strict.
These results already cover negative correlation for
connected spanning subgraphs of $K_{2,n}$ and $K_{3,3}$.

The paper is organized as follows. {\autoref{Sec-2} proves that
two-edge inclusion counts are largest for disjoint edge pairs, even with
the total number of edges fixed. Edge transitivity then reduces both main
theorems to these pairs.}
In \autoref{sec:forests}, we enumerate forests subject to
component restrictions by a signed count of directed graphs.
Edge-deletion and path-deletion bijections, together with recurrence
relations, 
give formulas
for the edge-inclusion counts. These formulas reduce the proof of
\autoref{thm:main-forest} to showing that an explicit expression is positive.
In \autoref{sec:connected}, we group connected spanning subgraphs by
their restriction outside the endpoints of two disjoint edges.
M\"obius inversion expresses a weighted sum of extension counts
in terms of products over the components of the restriction.
Product inequalities for its components yield an inequality for the extension counts with a
parameter independent of the fixed subgraph. Summing these
inequalities proves \autoref{thm:main-connected} for $m,n\geq3$.
The cases $m=2$ or $n=2$ follow from direct counting.
{Together with the comparisons in \autoref{Sec-2}, these estimates give
strict negative correlation between any two distinct edges of a fixed
copy of $K_{2,2}$ after conditioning on the subgraph induced by the
remaining vertices; see \autoref{cor:conditional-connected}.}

\section{Reduction to disjoint edge pairs}\label{Sec-2}
We reduce both main theorems to disjoint edge pairs by comparing joint
edge-inclusion counts for the two uniform measures.

Write $uv$ for the edge $\{u,v\}$.
For $U\subseteq V(G)$, the graph $G-U$ is obtained by deleting $U$ and
all incident edges.
The induced subgraph $G[U]$ has vertex set $U$ and all edges
of $G$ with both endpoints in $U$; the same convention applies to
digraphs, with arcs in place of edges.
For an edge set $X$ whose endpoints lie in $V(G)$, define
\[
G-X\coloneqq(V(G),E(G)\setminus X),\qquad G+X\coloneqq(V(G),E(G)\cup X);
\]    
we abbreviate $G\pm\{e\}$ to $G\pm e$.
Let $\Comp(G)$ denote the set of connected components of $G$, including
isolated vertices, and put $\omega(G)\coloneqq|\Comp(G)|$.
Two edges are disjoint if they have no common endpoint.
In subscripts, commas separate distinct edges written by their endpoints;
for example, $j_{ab,cd}(G)=j_{\{ab,cd\}}(G)$.

The automorphisms of $K_{m,n}$ that preserve $L$ and $R$ act transitively
on edges. Hence, 
\begin{equation}
i_e(K_{m,n}) = i_f(K_{m,n}) \quad \text{and} \quad j_e(K_{m,n}) = j_f(K_{m,n})
\quad \text{for all distinct edges } e,f.
\label{eq:equal-edge-counts}
\end{equation}
A pair of distinct edges either has a common
end in $L$, has a common end in $R$, or is disjoint.
Therefore, to reduce both main theorems to disjoint edge pairs, it suffices
to show that two-edge inclusion counts are largest for disjoint pairs.
We prove the stronger coefficientwise comparisons for forests
and connected spanning subgraphs; evaluating them at $t=1$ gives
the inclusion-count comparisons needed for the reduction.

For $X\subseteq E(K_{m,n})$, we consider
\[
 i_X(K_{m,n};t)\coloneqq\sum_{F\in\mathcal F_X(K_{m,n})}t^{|E(F)|},\qquad
 j_X(K_{m,n};t)\coloneqq\sum_{H\in\mathcal C_X(K_{m,n})}t^{|E(H)|}.
\]
For polynomials $f,g\in\mathbb R[t]$, write $f\succeq g$ if every
coefficient of $f-g$ is nonnegative.

\begin{proposition}\label{prop:twin-comparison}
Let $a,c\in L$ and $b,d\in R$, with $a\ne c$ and $b\ne d$. Then
\[
 i_{ab,cd}(K_{m,n};t)\succeq i_{ab,ad}(K_{m,n};t),\qquad
 i_{ab,cd}(K_{m,n};t)\succeq i_{ab,cb}(K_{m,n};t).
\]
\end{proposition}

\begin{proof}
Let $X$ be one of $\{ab,ad\}$, $\{ab,cb\}$, and $\{ab,cd\}$.
For $H\in\mathcal F(K_{m,n}-X)$, $H+X$ is a forest if and only if
each edge of $X$ joins distinct components of $H$ and the endpoints of
the two edges occupy at least three components of $H$.
Indeed, two edges joining distinct components create a cycle precisely
when they join the same pair of components, which is equivalent to their
endpoints occupying only two components.

These conditions depend only on which of $a,b,c,d$ lie in the same
component of $H$. Accordingly, for $F\in\mathcal F(K_{m,n})$, let
$\pi(F)$ be the
partition of $\{a,b,c,d\}$ in which two vertices lie in the same block
if and only if they belong to the same component of $F$. Put
\[
 \pi_0\coloneqq\{\{a\},\{b\},\{c\},\{d\}\},\qquad
 \pi_{xy}\coloneqq\{\{x,y\}\}\cup
          \{\{z\}:z\in\{a,b,c,d\}\setminus\{x,y\}\}
\]
for distinct $x,y\in\{a,b,c,d\}$, and define
\[
 \mathcal H_\pi\coloneqq\{F\in\mathcal F(K_{m,n}):\pi(F)=\pi\},\qquad
 h_\pi(t)\coloneqq\sum_{F\in\mathcal H_\pi}t^{|E(F)|}.
\]

Suppose $H+X$ is a forest. For either adjacent pair $X$, the three
endpoints of $X$ lie in distinct components of $H$.
The remaining vertex of $\{a,b,c,d\}$ may lie in one of these components
or in a fourth.
Thus, $\pi(H)\in\{\pi_0,\pi_{ac},\pi_{bc},\pi_{cd}\}$ for $X=\{ab,ad\}$,
and $\pi(H)\in\{\pi_0,\pi_{ad},\pi_{bd},\pi_{cd}\}$ for $X=\{ab,cb\}$.
For $X=\{ab,cd\}$, the two component conditions give
$\pi(H)\in\{\pi_0,\pi_{ac},\pi_{ad},\pi_{bc},\pi_{bd}\}$. Conversely, if $\pi$ is one of the partitions listed for $X$, then
every $H\in\mathcal H_\pi$ lies in $\mathcal F(K_{m,n}-X)$ and $H+X$
is a forest. Thus, deleting $X$ gives a bijection from $\mathcal F_X(K_{m,n})$
onto the disjoint union of $\mathcal H_\pi$ over the partitions $\pi$
listed for $X$, with inverse $H\mapsto H+X$.
For each partition $\pi$ listed for $X$, this bijection gives
\[
 \sum_{F\in\mathcal F_X(K_{m,n}),\pi(F-X)=\pi}t^{|E(F)|}
 =t^2h_\pi(t),
\]
since $|E(H+X)|=|E(H)|+2$ for every $H\in\mathcal H_\pi$.
The automorphisms interchanging $a,c$ and $b,d$ give
$h_{\pi_{ad}}(t)=h_{\pi_{cd}}(t)$ and
$h_{\pi_{bc}}(t)=h_{\pi_{cd}}(t)$, respectively.
Therefore,
\begin{equation}
 \begin{aligned}
 i_{ab,ad}(K_{m,n};t)&=t^2\bigl(h_{\pi_0}(t)+h_{\pi_{ac}}(t)+2h_{\pi_{cd}}(t)\bigr),\\
 i_{ab,cb}(K_{m,n};t)&=t^2\bigl(h_{\pi_0}(t)+h_{\pi_{bd}}(t)+2h_{\pi_{cd}}(t)\bigr),\\
 i_{ab,cd}(K_{m,n};t)&=t^2\bigl(h_{\pi_0}(t)+h_{\pi_{ac}}(t)+h_{\pi_{bd}}(t)+2h_{\pi_{cd}}(t)\bigr).
 \end{aligned}
 \label{eq:forest-component-counts}
\end{equation}
Hence,
\[
 i_{ab,cd}(K_{m,n};t)-i_{ab,ad}(K_{m,n};t)=t^2h_{\pi_{bd}}(t),\qquad
 i_{ab,cd}(K_{m,n};t)-i_{ab,cb}(K_{m,n};t)=t^2h_{\pi_{ac}}(t).
\]
Both right-hand sides have nonnegative coefficients.
\end{proof}

\begin{proposition}\label{prop:connected-pair-comparison}
Let $a,c\in L$ and $b,d\in R$, with $a\ne c$ and $b\ne d$. Then
\[
 j_{ab,cd}(K_{m,n};t)\succeq j_{ab,ad}(K_{m,n};t),\qquad
 j_{ab,cd}(K_{m,n};t)\succeq j_{ab,cb}(K_{m,n};t).
\]
\end{proposition}

\begin{proof}
For the first inequality, both edge pairs contain one edge incident with
$b$ and one incident with $d$. Therefore, we group connected spanning
subgraphs by their restriction after deleting $b,d$.
For each spanning subgraph $H$ of $K_{m,n}-\{b,d\}$ and
$u_1,u_2\in\{a,c\}$, define
\begin{equation}
 \varphi_{u_1,u_2}(H;t)\coloneqq\sum_{G\in\mathcal C(K_{m,n}),\ G-\{b,d\}=H,\ bu_1,du_2\in E(G)}t^{|E(G)|-|E(H)|}.
 \label{eq:phi-def}
\end{equation}
Every $G\in\mathcal C(K_{m,n})$ has a unique restriction
$H=G-\{b,d\}$. Summing over all spanning subgraphs $H$ of
$K_{m,n}-\{b,d\}$ gives
\[
 j_{ab,ad}(K_{m,n};t)=\sum_Ht^{|E(H)|}\varphi_{a,a}(H;t),\qquad
 j_{ab,cd}(K_{m,n};t)=\sum_Ht^{|E(H)|}\varphi_{a,c}(H;t).
\]
The automorphism interchanging $a,c$ permutes the restrictions $H$ and
preserves edge counts. Hence,
$\sum_Ht^{|E(H)|}\varphi_{a,a}(H;t)=
\sum_Ht^{|E(H)|}\varphi_{c,c}(H;t)$.
Averaging these equal expressions yields
\begin{equation}
 2\bigl(j_{ab,cd}(K_{m,n};t)-j_{ab,ad}(K_{m,n};t)\bigr)
 =\sum_H t^{|E(H)|}
        \bigl(2\varphi_{a,c}(H;t)-\varphi_{a,a}(H;t)-\varphi_{c,c}(H;t)\bigr).
 \label{eq:connected-pair-comparison}
\end{equation}
It suffices to prove
$2\varphi_{a,c}(H;t)-\varphi_{a,a}(H;t)-\varphi_{c,c}(H;t)\succeq0$
for every $H$.

Fix $H$. For $U\in\Comp(H)$ and $v\in\{b,d\}$, put
\[
 E_v(U)\coloneqq\{vu:u\in V(U)\cap L\},\qquad
 \ell_U\coloneqq|V(U)\cap L|.
\]
A spanning subgraph $G$ with $G-\{b,d\}=H$ has the form
$G=(L\cup R,E(H)\cup X_G)$ for an edge set
$X_G\subseteq\{bu,du:u\in L\}$. It is connected if and only if
\begin{equation}
 \begin{aligned}
 &X_G\cap\bigl(E_b(U)\cup E_d(U)\bigr)\ne\varnothing
       &&\text{for every }U\in\Comp(H),\\
 &X_G\cap E_b(U)\ne\varnothing,\quad X_G\cap E_d(U)\ne\varnothing
       &&\text{for some }U\in\Comp(H).
 \end{aligned}
 \label{eq:attachment-connectivity}
\end{equation}
If $G$ is connected, every component of $H$ must be joined to
$b$ or $d$. A $b$--$d$ path has its internal vertices in a single
component of $H$, since $b,d$ are not adjacent. Conversely, the first
condition attaches every component to $b$ or $d$, and the second
connects $b$ to $d$.

If $\ell_U=0$ for some $U\in\Comp(H)$, then $E_b(U)=E_d(U)=\varnothing$, so the
first condition cannot hold and all $\varphi_{u_1,u_2}(H;t)$ vanish.
Assume that every $\ell_U\geq1$.
If $a,c\in V(U)$ for some $U\in\Comp(H)$, the required edges
$bu_1,du_2$ satisfy the second condition in
\eqref{eq:attachment-connectivity} for all $u_1,u_2\in\{a,c\}$.
The remaining $2\ell_U-2$ edges of $E_b(U)\cup E_d(U)$ may be chosen
freely, and a nonempty subset of $E_b(T)\cup E_d(T)$ must be chosen
for each other component $T$. Hence,
\[
 \varphi_{a,a}(H;t)=\varphi_{a,c}(H;t)=\varphi_{c,c}(H;t)
 =t^2(1+t)^{2\ell_U-2}
   \prod_{T\in\Comp(H)\setminus\{U\}}\bigl((1+t)^{2\ell_T}-1\bigr).
\]

Now assume that $a,c$ lie in distinct components $U_a,U_c$.
For each other component $U$, the generating polynomial for nonempty
subsets of $E_b(U)\cup E_d(U)$ is $(1+t)^{2\ell_U}-1$.
For nonempty subsets using edges incident with exactly one
of $b,d$, the polynomial is $2((1+t)^{\ell_U}-1)$.
Distinct components use disjoint edge sets, so their generating
polynomials multiply to give
\[
 \psi(t)\coloneqq\prod_{U\in\Comp(H)\setminus\{U_a,U_c\}}
             \bigl((1+t)^{2\ell_U}-1\bigr),\qquad
 \psi^{\mathrm{sep}}(t)\coloneqq\prod_{U\in\Comp(H)\setminus\{U_a,U_c\}}
             2\bigl((1+t)^{\ell_U}-1\bigr),
\]
where an empty product is taken to be $1$.

For $\varphi_{a,a}(H;t)$, the edges $ba,da$ required
in~\eqref{eq:phi-def} satisfy the second
condition in~\eqref{eq:attachment-connectivity}. The remaining edges of
$E_b(U_a)\cup E_d(U_a)$ are unrestricted, and a nonempty subset of
$E_b(U_c)\cup E_d(U_c)$ must be chosen.
The same reasoning gives $\varphi_{c,c}(H;t)$.
For $\varphi_{a,c}(H;t)$, fix $ba,dc\in X_G$ as in~\eqref{eq:phi-def}.
Among edge sets $X_G$ satisfying the first condition, the
second fails exactly when no edge from either $E_d(U_a)$ or $E_b(U_c)$
is chosen and every other component $U$ uses a nonempty subset of
exactly one of $E_b(U)$ and $E_d(U)$.
Including the factor $t^2$ for the two fixed edges gives

\[
 \begin{aligned}
 \varphi_{a,a}(H;t)&=t^2(1+t)^{2\ell_{U_a}-2}
                 \bigl((1+t)^{2\ell_{U_c}}-1\bigr)\psi(t),\\
 \varphi_{c,c}(H;t)&=t^2(1+t)^{2\ell_{U_c}-2}
                 \bigl((1+t)^{2\ell_{U_a}}-1\bigr)\psi(t),\\
 \varphi_{a,c}(H;t)&=t^2(1+t)^{2\ell_{U_a}+2\ell_{U_c}-2}\psi(t)
          -t^2(1+t)^{\ell_{U_a}+\ell_{U_c}-2}\psi^{\mathrm{sep}}(t).
 \end{aligned}
\]
Subtracting the first two lines from twice the third yields
\begin{align*}
 &2\varphi_{a,c}(H;t)-\varphi_{a,a}(H;t)-\varphi_{c,c}(H;t)\\&=t^2\bigl[ ((1+t)^{\ell_{U_a}-1}-(1+t)^{\ell_{U_c}-1})^2\psi(t)+2(1+t)^{\ell_{U_a}+\ell_{U_c}-2} (\psi(t)-\psi^{\mathrm{sep}}(t))\bigr].
\end{align*}
For every $U\in\Comp(H)\setminus\{U_a,U_c\}$,
\[
 ((1+t)^{2\ell_U}-1)-2((1+t)^{\ell_U}-1)
       =((1+t)^{\ell_U}-1)^2\succeq0.
\]
All factors in both products have nonnegative coefficients.
Hence, $\psi(t)\succeq\psi^{\mathrm{sep}}(t)$. Moreover,
\[
 ((1+t)^{\ell_{U_a}-1}-(1+t)^{\ell_{U_c}-1})^2
 =(1+t)^{2(\min\{\ell_{U_a},\ell_{U_c}\}-1)}
      ((1+t)^{|\ell_{U_a}-\ell_{U_c}|}-1)^2\succeq0.
\]
Thus, $2\varphi_{a,c}(H;t)-\varphi_{a,a}(H;t)-\varphi_{c,c}(H;t)\succeq0$ for every $H$.
The first inequality follows by summing over $H$.
The same argument with the roles of $L$ and $R$ exchanged proves the second.
\end{proof}

{At $t=1$, the preceding comparisons and edge transitivity give
the following reduction.}

\begin{corollary}\label{cor:pair-order}
Let $e,f$ be distinct edges and $e',f'$ disjoint edges of
$K_{m,n}$, where $m,n\geq2$. Then
\[
 \begin{aligned}
 i_e(K_{m,n})i_f(K_{m,n})-i_{ef}(K_{m,n})i(K_{m,n})
 &\geq i_{e'}(K_{m,n})i_{f'}(K_{m,n})-i_{e'f'}(K_{m,n})i(K_{m,n}),\\
 j_e(K_{m,n})j_f(K_{m,n})-j_{ef}(K_{m,n})j(K_{m,n})
 &\geq j_{e'}(K_{m,n})j_{f'}(K_{m,n})-j_{e'f'}(K_{m,n})j(K_{m,n}).
 \end{aligned}
\]
\end{corollary}

\begin{proof}
Automorphisms of $K_{m,n}$ obtained by independently permuting
$L$ and $R$ act transitively on edges and on disjoint edge pairs,
and preserve forests and connected subgraphs.
Thus, the one-edge inclusion counts are equal for all edges, 
and all disjoint edge pairs have the same two-edge inclusion counts.
By \autoref{prop:twin-comparison} and
\autoref{prop:connected-pair-comparison} at $t=1$, the two-edge inclusion
counts are largest for disjoint pairs.
Since $i(K_{m,n}),j(K_{m,n})\geq0$, the assertion follows.
\end{proof}

It remains to prove strict negative correlation for one disjoint
edge pair under each measure. We begin with the forest measure.

\section{The uniform forest measure}\label{sec:forests}

We prove \autoref{thm:main-forest} by computing $i(K_{m,n})$,
$i_e(K_{m,n})$, and $i_{ef}(K_{m,n})$ for disjoint edges $e,f$.
Deleting an edge from a forest separates its endpoints. This
suggests counting forests in which each component contains at most one
vertex of a prescribed set. The component-partition identities in
\autoref{prop:twin-comparison} then express the disjoint-edge count
in terms of these forests.
We enumerate these forests by a signed count of directed graphs.

\subsection{Forest enumeration}

Throughout this section, fix $m,n\geq2$.
For $W\subseteq L\cup R$, define
\[
 \mathcal F_{\leq 1}(W)\coloneqq\{F\in\mathcal F(K_{m,n}):
    |V(U)\cap W|\leq1\text{ for every }U\in\Comp(F)\}.
\]
Put $p=|L\setminus W|$, $q=|R\setminus W|$, and
$f_{p,q}\coloneqq|\mathcal F_{\leq 1}(W)|$. Permutations within the two parts show that
this count depends only on $p,q$ while $m,n$ remain fixed.

\begin{proposition}\label{lem:elementary-forest-enumeration}
For integers $p,q$ with $0\leq p\leq m$ and $0\leq q\leq n$,
\[
 f_{p,q}=\sum_{k=0}^{\min\{p,q\}}(-1)^k(k+1)
             \binom pk\binom qk k!(n+1)^{p-k}(m+1)^{q-k}.
\]
\end{proposition}

The proof starts from the identity $|V(T)|-|E(T)|=1$ for a tree $T$.
For each component disjoint from $W$, we count its vertices with weight
$1$ and its edges with weight $-1$. A directed graph will record the
vertex or edge selected in each component.
For every $W\subseteq L\cup R$, let $\mathfrak D_W$ consist of the
directed graphs $D$ on $L\cup R$ satisfying
\[
 E(D)\subseteq(L\times R)\cup(R\times L),\qquad
 d_D^+(v)\leq1\ (v\in L\cup R),\quad d_D^+(w)=0\ (w\in W),
\]
where $d_D^+$ denotes outdegree.
For each edge $uv$ of $K_{m,n}$, $D$ may contain neither of the arcs
$(u,v),(v,u)$, exactly one, or both, subject to these outdegree conditions.
If both occur, they form a directed cycle of length $2$.
Let $\underline D$ be the simple graph on $L\cup R$ with edge set
$\{uv:(u,v)\in E(D)\text{ or }(v,u)\in E(D)\}$.
Thus, a pair of opposite arcs in $D$ gives a single edge of $\underline D$.
A vertex of outdegree zero, or \emph{sink}, records a selected vertex;
a $2$-cycle records a selected edge.
We give longer cycles weight zero to exclude cyclic underlying graphs.
For $D\in\mathfrak D_W$, let $\mathfrak C(D)$ denote its set of
directed cycles and define
\[
 w(D)\coloneqq\prod_{C\in\mathfrak C(D)}\varepsilon(C),\qquad
 \varepsilon(C)\coloneqq
 \begin{cases}
 -1,&|E(C)|=2,\\
 0,&|E(C)|>2.
 \end{cases}
\]
\begin{lemma}\label{prop:forest-weighted-sum}
For every $W\subseteq L\cup R$, with $p=|L\setminus W|$ and
$q=|R\setminus W|$,
\[
 f_{p,q}=\sum_{D\in\mathfrak D_W}w(D).
\]
\end{lemma}

\begin{proof}
Fix $D\in\mathfrak D_W$ with $w(D)\ne0$.
From any vertex, following outgoing arcs leads either to a
sink or to a directed cycle.
The endpoints of each arc lead to the same sink or cycle.
Thus, for each $T\in\Comp(\underline D)$, the subdigraph $D[V(T)]$ 
induced by $V(T)$ has a unique sink or a unique
directed cycle, which has length $2$ since $w(D)\ne0$.
There is no pair of opposite arcs in the first case and exactly one
in the second. Hence,
\[
 \begin{aligned}
 |E(T)|&=|E(D[V(T)])|=|V(T)|-1
       &&\text{(a unique sink)},\\
 |E(T)|&=|E(D[V(T)])|-1=|V(T)|-1
       &&\text{(a unique $2$-cycle)}.
 \end{aligned}
\]
Thus, $T$ is a tree. Since every vertex of $W$ is a sink,
$|V(T)\cap W|\leq1$. Therefore,
$\underline D\in\mathcal F_{\leq1}(W)$.

Conversely, fix $F\in\mathcal F_{\leq1}(W)$.
We construct all $D\in\mathfrak D_W$ with $\underline D=F$ and
$w(D)\ne0$ by specifying the arcs on each $T\in\Comp(F)$.
If $V(T)\cap W=\{v\}$, direct all edges towards $v$, giving weight $1$.
If $T$ is disjoint from $W$, either choose a vertex
and direct all edges towards it, with weight $1$;
or choose an edge $uv$, replace it by opposite arcs, and
direct all edges of $T-uv$ towards $u$ and $v$ in their respective
components, with weight $-1$.
Each possible directed graph on $T$ with nonzero weight arises exactly once.
The constructions on distinct components are independent, and their weights
multiply. Hence,
\[
 \sum_{D\in\mathfrak D_W,\, \underline D=F}w(D)
 =\prod_{T\in\Comp(F)}
 \begin{cases}
  1,&V(T)\cap W\ne\varnothing,\\
  |V(T)|-|E(T)|,&V(T)\cap W=\varnothing
 \end{cases}
 =1.
\]
Summing over $F\in\mathcal F_{\leq1}(W)$ proves the lemma,
because the directed graphs with $w(D)\ne0$ all have underlying graphs
in this family.
\end{proof}

\begin{proof}[Proof of \autoref{lem:elementary-forest-enumeration}]
Choose $W\subseteq L\cup R$ with $|L\setminus W|=p$ and $|R\setminus W|=q$. 
Expanding each factor of $w(D)$ as $\varepsilon(C)=1+(\varepsilon(C)-1)$ and
interchanging the sums in \autoref{prop:forest-weighted-sum} gives
\[
 f_{p,q}
  =\sum_{D\in\mathfrak D_W}
 \sum_{\mathfrak Q\subseteq\mathfrak C(D)}
  \prod_{C\in\mathfrak Q}(\varepsilon(C)-1)=\sum_{\mathfrak Q}
 \Big(\prod_{C\in\mathfrak Q}(\varepsilon(C)-1)\Big)
 \sum_{D\in\mathfrak D_W,\,\mathfrak Q\subseteq\mathfrak C(D)}1,
\]
where $\mathfrak Q$ ranges over all families of directed cycles with arcs
in $(L\times R)\cup(R\times L)$.
By the outdegree conditions, the inner sum is nonzero if and only if
the cycles in $\mathfrak Q$ are pairwise vertex-disjoint and avoid $W$.
Fix one such family $\mathfrak Q$. Its cycles alternate between $L$ and $R$,
so together they use $k$ vertices from each part.
The inclusion $\mathfrak Q\subseteq\mathfrak C(D)$ fixes the outgoing arc
at every vertex on its cycles but allows additional cycles in $D$.
Each of the $p-k$ remaining vertices in $L\setminus W$ may have no outgoing
arc or one to any vertex of $R$, giving $n+1$ choices; similarly, each of
the $q-k$ remaining vertices in $R\setminus W$ has $m+1$ choices.
These choices are independent, so the inner sum equals $(n+1)^{p-k}(m+1)^{q-k}$.

Fix an ordering $u_1,\ldots,u_k$ of the vertices of $\mathfrak Q$ in $L$.
The arcs from $L$ to $R$ form a bijection;
write $v_i$ for the image of $u_i$.
The arcs from $R$ to $L$ form another bijection.
Write $\mathcal S_k$ for the symmetric group on
$[k]\coloneqq\{1,\ldots,k\}$, with $[0]=\varnothing$ and
$\mathcal S_0$ consisting of the empty permutation.
The second bijection determines a unique $\sigma\in\mathcal S_k$ such that
\[
 \bigcup_{C\in\mathfrak Q}E(C)
  =\{(u_i,v_i),(v_i,u_{\sigma(i)}):1\leq i\leq k\}.
\]
Conversely, for the fixed vertex sets and first bijection
$u_i\mapsto v_i$, every $\sigma\in\mathcal S_k$ defines a unique family
$\mathfrak Q$ of vertex-disjoint directed cycles by this formula.
A cycle of length $\ell$ in $\sigma$ corresponds to a directed cycle $C$
of length $2\ell$, so $\varepsilon(C)-1$ equals $-2$ for $\ell=1$ and
$-1$ for $\ell>1$.
Let $\operatorname{fix}(\sigma)$ be the number of fixed points of $\sigma$.
Since an $\ell$-cycle has sign $(-1)^{\ell-1}$, we obtain
\[
 \prod_{C\in\mathfrak Q}(\varepsilon(C)-1)
 =(-1)^k\operatorname{sgn}(\sigma)2^{\operatorname{fix}(\sigma)}.
\]
To sum over $\mathfrak Q$, we choose $k$ vertices from $L\setminus W$
and $k$ from $R\setminus W$ in $\binom pk\binom qk$ ways.
The first bijection then has $k!$ choices.
Combining these counts and summing over $k$ gives
\[
 f_{p,q}
 =\sum_{k=0}^{\min\{p,q\}}(-1)^k
    \binom pk\binom qk k!(n+1)^{p-k}(m+1)^{q-k}
    \sum_{\sigma\in\mathcal S_k}
       \operatorname{sgn}(\sigma)2^{\operatorname{fix}(\sigma)}.
\]
Since $2^{\operatorname{fix}(\sigma)}$ counts subsets of these fixed points,
for every integer $k\geq0$,
\begin{equation}
 \sum_{\sigma\in\mathcal S_k}
       \operatorname{sgn}(\sigma)2^{\operatorname{fix}(\sigma)}
 =\sum_{J\subseteq[k]}
 \sum_{\sigma\in\mathcal S_k, \sigma(j)=j\ \text{for all }j\in J}
       \operatorname{sgn}(\sigma)
 =k+1.
 \label{eq:signed-permutation-sum}
\end{equation}
If $|[k]\setminus J|\geq2$, choose distinct $k_1,k_2\in[k]\setminus J$.
The pairing $\sigma\leftrightarrow(k_1\,k_2)\circ\sigma$ cancels the terms
in the inner sum. Otherwise, only the identity contributes, so each of
the $k+1$ subsets $J$ with $|[k]\setminus J|\leq1$ contributes $1$.
Substituting~\eqref{eq:signed-permutation-sum} proves the proposition.
\end{proof}

\subsection{Edge-inclusion counts and negative correlation}

To compute the disjoint-edge count via
\eqref{eq:forest-component-counts}, we must count forests in which the
endpoints of one prescribed edge share a component and the endpoints of
the other lie in two further components. Deleting the path joining the
endpoints of the first edge expresses this count as a sum of restricted
forest counts. The auxiliary sum below evaluates this count.

For integers $p,q$ with $0\leq p\leq m$ and $0\leq q\leq n$,
define the auxiliary sum obtained from
\autoref{lem:elementary-forest-enumeration} by omitting the factor $k+1$:
\[
g_{p,q}\coloneqq\sum_{k=0}^{\min\{p,q\}}(-1)^k\binom pk\binom qk k!
(n+1)^{p-k}(m+1)^{q-k}.
\]
Set $g_{p,q}=f_{p,q}=0$ when $\min\{p,q\}<0$, and use
$\binom{\ell}{k}=0$ for integers $k>\ell\geq0$.
The identity
\begin{equation}
\ell\binom{\ell-1}{k-1}=k\binom{\ell}{k}\qquad(\ell,k\geq1)
\label{eq:binomial-identity}
\end{equation}
and an index shift separate the contributions of $1$ and $k$
in \autoref{lem:elementary-forest-enumeration}, yielding
\begin{equation}
 f_{p,q}=g_{p,q}-pqg_{p-1,q-1}.
 \label{eq:forest-auxiliary}
\end{equation}
When one index is zero, both sums reduce to
$g_{p,0}=f_{p,0}=(n+1)^p$ or
$g_{0,q}=f_{0,q}=(m+1)^q$, so~\eqref{eq:forest-auxiliary}
also holds on the boundary.
For $0\leq p\leq m$ and $0\leq q\leq n$, the edgeless graph is counted by
$f_{p,q}$, so $f_{p,q}\geq1$.
Iterating~\eqref{eq:forest-auxiliary} until one index is zero gives
\begin{equation}
 g_{p,q}
 =\sum_{k=0}^{\min\{p,q\}}\binom pk\binom qk(k!)^2f_{p-k,q-k}
 \geq f_{p,q}\geq1.
 \label{eq:forest-positive-expansion}
\end{equation}

We also record the recurrences used below to evaluate the inclusion counts
and compare $f_{p,q}$ with $g_{p,q}$.
For $0\leq p\leq m$ and $0\leq q\leq n$, we have
\begin{subequations}
\begin{align}
& g_{p,q}=(n+1)g_{p-1,q}-qg_{p-1,q-1}
     &&(p\geq1),\label{eq:forest-recurrence-left}\\
&g_{p,q}=(m+1)g_{p,q-1}-pg_{p-1,q-1}
     &&(q\geq1),\label{eq:forest-recurrence-right}\\
& g_{p+1,q+1}=(mn+m+n-p-q-1)g_{p,q}+f_{p,q}
     &&(p<m,\ q<n).\label{eq:forest-recurrence-diagonal}
\end{align}
\end{subequations}
Expand $\binom pk$ by Pascal's identity in the definition of
$g_{p,q}$. Reindexing the second sum with~\eqref{eq:binomial-identity}
proves~\eqref{eq:forest-recurrence-left};
\eqref{eq:forest-recurrence-right} follows by interchanging the parts.
For~\eqref{eq:forest-recurrence-diagonal}, assume first that $p,q\geq1$.
Substitute
\eqref{eq:forest-recurrence-right} into
\eqref{eq:forest-recurrence-left} at $(p+1,q+1)$.
Then eliminate $(n+1)g_{p-1,q}$ using
\eqref{eq:forest-recurrence-left} and replace $pqg_{p-1,q-1}$ by
$g_{p,q}-f_{p,q}$ using~\eqref{eq:forest-auxiliary}.
The zero-index cases follow directly.

\begin{lemma}
\label{lem:three-forest-counts}
Let $e,f$ be disjoint edges of $K_{m,n}$, and put $r=m-2$ and $s=n-2$. Then
\[
 i(K_{m,n})=(mn+m+n+5)g_{r,s}+2f_{r,s},\qquad
 i_e(K_{m,n})=(m+n+2)g_{r,s}+f_{r,s}.
\]
\[
 (m+1)(n+1)i_{ef}(K_{m,n})
 =(m+n)(m+n+2)g_{r,s}+(mn-1)f_{r,s}.
\]
\end{lemma}

\begin{proof}
Write $e=ab$ and $f=cd$, where $a,c\in L$ and $b,d\in R$.
Deleting the prescribed edges gives
\[
 \begin{aligned}
 \mathcal F_{\{ab\}}(K_{m,n})
   &\longleftrightarrow\mathcal F_{\leq1}(\{a,b\}),\\
 \mathcal F_{\{ab,ad\}}(K_{m,n})
   &\longleftrightarrow\mathcal F_{\leq1}(\{a,b,d\}),\\
 \mathcal F_{\{ab,cb\}}(K_{m,n})
   &\longleftrightarrow\mathcal F_{\leq1}(\{a,b,c\}).
 \end{aligned}
\]
In each case, adding the prescribed edges joins distinct components
and is the inverse. Thus, $i(K_{m,n})=f_{m,n}$, $i_e(K_{m,n})=f_{m-1,n-1}$,
$i_{ab,ad}(K_{m,n})=f_{r+1,s}$, and $i_{ab,cb}(K_{m,n})=f_{r,s+1}$.
Recall from the proof of \autoref{prop:twin-comparison} that
$\mathcal H_\pi$ consists of forests inducing the component partition
$\pi$ on $\{a,b,c,d\}$, and $h_\pi(1)=|\mathcal H_\pi|$.
The partition $\pi_0$ separates all four vertices, whereas $\pi_{cd}$
puts only $c,d$ together.
Setting $t=1$ in~\eqref{eq:forest-component-counts} gives
\[
 i_{ef}(K_{m,n})=i_{ab,ad}(K_{m,n})+i_{ab,cb}(K_{m,n})-h_{\pi_0}(1)-2h_{\pi_{cd}}(1).
\]
Here, $\mathcal H_{\pi_0}=\mathcal F_{\leq1}(\{a,b,c,d\})$,
so $h_{\pi_0}(1)=f_{r,s}$.
To count $\mathcal H_{\pi_{cd}}$, delete the edges of the unique
$c$--$d$ path $P_F$ in each $F\in\mathcal H_{\pi_{cd}}$. The map
\[
 F\longmapsto(P_F,F-E(P_F))
\]
is a bijection onto pairs $(P,F')$ with $P$ a $c$--$d$
path in $K_{m,n}-\{a,b\}$ and
$F'\in\mathcal F_{\leq1}(V(P)\cup\{a,b\})$.
Indeed, deleting the path edges separates its vertices into distinct
components, since a forest has unique paths. Conversely, adding
$E(P)$ joins those components into a tree with unique $c$--$d$
path $P$ and leaves $a,b$ in two other components.
Thus, $(P,F')\mapsto F'+E(P)$ is the inverse.

A $c$--$d$ path of length $2k+1$ has $k$ ordered internal vertices in each part.
There are $\binom rk\binom sk(k!)^2$ such paths, and for each path the
number of possible $F'$ is $f_{r-k,s-k}$.
The expansion~\eqref{eq:forest-positive-expansion}
identifies the resulting path-length sum:
\[
 h_{\pi_{cd}}(1)=\sum_{k=0}^{\min\{r,s\}}
          \binom rk\binom sk(k!)^2 f_{r-k,s-k}=g_{r,s}.
\]
Thus, the three inclusion counts satisfy
\begin{equation}
 i(K_{m,n})=f_{m,n},\quad i_e(K_{m,n})=f_{m-1,n-1},\quad
 i_{ef}(K_{m,n})=f_{r+1,s}+f_{r,s+1}-f_{r,s}-2g_{r,s}.
 \label{eq:three-forest-counts}
\end{equation}

Apply~\eqref{eq:forest-recurrence-diagonal} at $(r,s)$ and
$(m-1,n-1)$. In the second instance, eliminate $g_{m,n}$ using
\eqref{eq:forest-auxiliary}:
\[
 g_{m-1,n-1}=(mn+3)g_{r,s}+f_{r,s},\qquad f_{m,n}=g_{m-1,n-1}+f_{m-1,n-1}.
\]
Substituting for $f_{m-1,n-1}$ from
\eqref{eq:forest-auxiliary} evaluates the first two counts:
\[
 i(K_{m,n})=g_{m-1,n-1}+f_{m-1,n-1}=(mn+m+n+5)g_{r,s}+2f_{r,s},
 \]
 \[
 i_e(K_{m,n})=g_{m-1,n-1}-(m-1)(n-1)g_{r,s}=(m+n+2)g_{r,s}+f_{r,s}.
\]
At $(p,q)=(r,s)$, \eqref{eq:forest-auxiliary} reads
\begin{equation}
 rsg_{r-1,s-1}=g_{r,s}-f_{r,s}.
 \label{eq:forest-auxiliary-rs}
\end{equation}
For $f_{r+1,s}$, substitute
\eqref{eq:forest-recurrence-left} into~\eqref{eq:forest-auxiliary},
multiply by $m+1$, and apply
\eqref{eq:forest-recurrence-right}:
\[
 \begin{aligned}
 (m+1)f_{r+1,s}
 &=(m+1)(n+1)g_{r,s}-(r+2)s(m+1)g_{r,s-1}\\
 &=(m+1)(n+1)g_{r,s}
       -(r+2)(sg_{r,s}+rsg_{r-1,s-1}).
 \end{aligned}
\]
For $s=0$, the same equality follows directly from the boundary
values. Equation~\eqref{eq:forest-auxiliary-rs} replaces
$rsg_{r-1,s-1}$ by $g_{r,s}-f_{r,s}$.
{Interchanging the parts gives the corresponding formula for
$f_{r,s+1}$:}
\[
 f_{r+1,s}=\frac{(2m+n+1)g_{r,s}+mf_{r,s}}{m+1},\qquad
 f_{r,s+1}=\frac{(m+2n+1)g_{r,s}+nf_{r,s}}{n+1}.
\]
The formula for $i_{ef}(K_{m,n})$ now follows from
\eqref{eq:three-forest-counts}.
\end{proof}
\begin{proof}[Proof of \autoref{thm:main-forest}]
By \autoref{cor:pair-order}, it suffices to consider disjoint edges
$e,f$. Set $r=m-2$ and $s=n-2$.
Substituting the counts from \autoref{lem:three-forest-counts}
and using~\eqref{eq:equal-edge-counts} yields the identity
\begin{align}
 &(m+1)(n+1)\bigl(i_e(K_{m,n})i_f(K_{m,n})-i(K_{m,n})i_{ef}(K_{m,n})\bigr)\notag\\
 =& (m-1)(n-1)
       \bigl(2(m+n+2)g_{r,s}^2-(mn+3)g_{r,s}f_{r,s}-f_{r,s}^2\bigr)
 +4(mn+3)g_{r,s}f_{r,s}+4f_{r,s}^2.
 \label{eq:forest-gap-reduction}
\end{align}
To bound the bracket, we need a sharper comparison between
$f_{r,s}$ and $g_{r,s}$. For $r,s\geq1$,
\eqref{eq:forest-recurrence-diagonal} at $(r-1,s-1)$ and
\eqref{eq:forest-positive-expansion} give
\[
 g_{r,s}=(mn+5)g_{r-1,s-1}+f_{r-1,s-1}\leq(mn+6)g_{r-1,s-1}.
\]
Since $mn+6-rs=2(m+n+1)$, multiply by $rs$ and use
\eqref{eq:forest-auxiliary-rs} to obtain
\[
 1\leq f_{r,s}\leq g_{r,s},\qquad
 (mn+6)f_{r,s}\leq2(m+n+1)g_{r,s}.
\]
If $r=0$ or $s=0$, the same bounds hold because $g_{r,s}=f_{r,s}$
and $2(m+n+1)=mn+6$.
The inequality $f_{r,s}^2\leq g_{r,s}f_{r,s}$ bounds the
subtracted terms by $(mn+4)g_{r,s}f_{r,s}$. Applying the sharper
comparison above gives
\[
 2(m+n+2)g_{r,s}^2-(mn+3)g_{r,s}f_{r,s}-f_{r,s}^2
 \geq 2g_{r,s}^2>0.
\]
Therefore, the right-hand side of
\eqref{eq:forest-gap-reduction} is positive, proving the theorem.
\end{proof}
\section{The uniform connected-subgraph measure}\label{sec:connected}
By \autoref{cor:pair-order}, it suffices to consider two disjoint edges.
We first assume $m,n\geq3$; the cases in which one part has
size two are treated in the final proof.

\subsection{Counting connected extensions}

Fix $a,c\in L$ and $b,d\in R$, with $a\ne c$ and $b\ne d$.
Set $e=ab$, $f=cd$, and
$W\coloneqq\{a,b,c,d\}$.
As in the proof of \autoref{prop:connected-pair-comparison}, we group
connected spanning subgraphs by their restriction after deleting $W$.
For a spanning subgraph $H$ of $K_{m,n}-W$, let $j^H$ count the
members $G$ of $\mathcal C(K_{m,n})$ with $G-W=H$.
Let $j_e^H$, $j_f^H$, and $j_{ef}^H$ count those containing $e$, $f$,
and both edges, respectively.
Every $G\in\mathcal C(K_{m,n})$ has a unique restriction $H=G-W$.
Summing over all spanning subgraphs $H$ of $K_{m,n}-W$ gives
\[
 j(K_{m,n})=\sum_Hj^H,\qquad
 j_e(K_{m,n})=\sum_Hj_e^H,\qquad
 j_{ef}(K_{m,n})=\sum_Hj_{ef}^H.
\]
Simultaneously interchanging $a,c$ and $b,d$ fixes $H$
and exchanges $e,f$, so $j_e^H=j_f^H$.

{By \eqref{eq:equal-edge-counts},
$j_e(K_{m,n})=j_f(K_{m,n})$.
Completing the square gives, for every real $\theta$,}

\begin{equation}\label{eq:centering-identity}
\begin{aligned}
 &j_e(K_{m,n})j_f(K_{m,n})-j(K_{m,n})j_{ef}(K_{m,n})\\
& =\bigl(j_e(K_{m,n})-\theta j(K_{m,n})\bigr)^2-j(K_{m,n})\bigl(j_{ef}(K_{m,n})-2\theta j_e(K_{m,n}) +\theta^2j(K_{m,n})\bigr).
\end{aligned}
\end{equation}
Since $j(K_{m,n})>0$, it suffices to find one parameter $\theta$,
depending only on $m,n$, for which
$j_{ef}^H-2\theta j_e^H+\theta^2j^H<0$ for every $H$.
We derive a componentwise expansion.

Fix $H$. List its isolated vertices in $L$ as $u_1,\ldots,u_l$,
those in $R$ as $v_1,\ldots,v_r$, and its remaining components as
$U_1,\ldots,U_N$. Thus, $N=\omega(H)-l-r\geq0$.
For these components, put
\[
 x_i\coloneqq2^{|V(U_i)\cap R|},\qquad
 y_i\coloneqq2^{|V(U_i)\cap L|}\qquad(1\leq i\leq N).
\]
Thus, $x_i$ counts the edge subsets joining $U_i$ to $a$ (or $c$),
and $y_i$ counts those joining $U_i$ to $b$ (or $d$),
including the empty set in both cases.
Each isolated vertex has three nonempty attachments to $W$.
We normalize these counts by $3$; for each partition below, an isolated
vertex then contributes $1$ or $2/3$. For a nonzero real parameter
$\theta$, set
\[
 \alpha\coloneqq(2/3)^l,\qquad
 \beta\coloneqq(2/3)^r,\qquad
 z\coloneqq4-\frac2\theta,\qquad
 \Phi\coloneqq\frac{j_{ef}^H-2\theta j_e^H+\theta^2j^H}
                         {\theta^2 3^{l+r}}.
\]
We compute $\Phi$ directly by weighted M\"obius inversion on the lattice
$\Pi(W)$ of partitions of $W$, ordered by refinement:
$\sigma\leq\pi$ means that each block of $\sigma$ is contained in a
block of $\pi$.
Write $\hat{1}=\{W\}$ for the one-block partition.

To obtain a product over components, we count nonempty attachments
from each component of $H$ to a single block of a partition.
For $\pi\in\Pi(W)$, put
\[
 q_\pi(x,y)\coloneqq
 \sum_{B\in\pi}\bigl(x^{|B\cap L|}y^{|B\cap R|}-1\bigr).
\]
Let $k(\pi)$ be the number of edges of $K_{m,n}[W]$
whose endpoints lie in a common block of $\pi$.
The products for the different partition types are defined in
Table~\ref{tab:partition-products}.
Vertical bars separate blocks.
All products run over $1\leq i\leq N$, with an empty product
taken to be $1$.
The last column gives the coefficient of each product in $\Phi$,
summed over the partitions in that row.

\begin{center}
\begin{minipage}{\linewidth}
\centering
\captionof{table}{Partitions of $W$, component products, and their coefficients in $\Phi$.}
\label{tab:partition-products}
\(
 \begin{array}{@{}cclc@{}}
 \hline
 \text{partitions }\pi&k(\pi)&\qquad\prod_i q_\pi(x_i,y_i)&
 \text{coefficient in }\Phi\\
 \hline
 abcd                           &4&
 \rho_1\coloneqq\prod\nolimits_i(x_i^2y_i^2-1)&z^2\\
 a|bcd,\ c|abd                   &2&
 \rho_{2,1}\coloneqq\prod\nolimits_i\bigl(x_i(y_i^2+1)-2\bigr)&-2z\beta\\
 b|acd,\ d|abc                   &2&
 \rho_{2,2}\coloneqq\prod\nolimits_i\bigl(y_i(x_i^2+1)-2\bigr)&-2z\alpha\\
 ac|bd                          &0&
 \rho_{2,3}\coloneqq\prod\nolimits_i(x_i^2+y_i^2-2)&-1\\
 ab|cd,\ ad|bc                   &2&
 \rho_{2,4}\coloneqq\prod\nolimits_i2(x_iy_i-1)&-\alpha\beta(4+z^2/4)\\
 ac|b|d                         &0&
 \rho_{3,1}\coloneqq\prod\nolimits_i(x_i^2+2y_i-3)&2\alpha\\
 bd|a|c                         &0&
 \rho_{3,2}\coloneqq\prod\nolimits_i(y_i^2+2x_i-3)&2\beta\\
 ab|c|d,\ ad|b|c,\ bc|a|d,\ cd|a|b&1&
 \rho_{3,0}\coloneqq\prod\nolimits_i(x_iy_i+x_i+y_i-3)&\alpha\beta(8+2z)\\
 a|b|c|d                        &0&
 \rho_4\coloneqq\prod\nolimits_i2(x_i+y_i-2)&-6\alpha\beta\\
 \hline
 \end{array}
\)
\end{minipage}
\end{center}

In the numerical subscripts, the first entry is the number of blocks
in a partition $\pi\in\Pi(W)$; the second distinguishes partition types.

\begin{lemma}\label{lem:four-endpoint-counts}
Let $H$ be a spanning subgraph of $K_{m,n}-W$, where $m,n\geq2$.
For every nonzero real $\theta$, with $z=4-2/\theta$, we have
\[
 \Phi=
    z^2\Bigl(\rho_1-\frac{\alpha\beta}{4}\rho_{2,4}\Bigr)
    -2z\bigl(\beta\rho_{2,1}+\alpha\rho_{2,2}-\alpha\beta\rho_{3,0}\bigr)+Q(\alpha,\beta),
\]
where
\[
 Q(\alpha,\beta)\coloneqq
    2\alpha\rho_{3,1}+2\beta\rho_{3,2}-\rho_{2,3}
    +\alpha\beta(8\rho_{3,0}-4\rho_{2,4}-6\rho_4).
\]
\end{lemma}

\begin{proof}
To obtain products over components, we represent $\Phi$
as a weighted sum. For each spanning subgraph $G$ of
$K_{m,n}$, define
\[
 w(G)\coloneqq\frac1{3^{l+r}}
       \Bigl(1-\frac{\mathbf1_{\{e\in E(G)\}}}{\theta}\Bigr)
       \Bigl(1-\frac{\mathbf1_{\{f\in E(G)\}}}{\theta}\Bigr).
\]
Expanding the product and using $j_e^H=j_f^H$, we deduce
\[
 \sum_{G\in\mathcal C(K_{m,n}),\,G-W=H}w(G)
 =\frac{j_{ef}^H-2\theta j_e^H+\theta^2j^H}{\theta^2 3^{l+r}}
 =\Phi.
\]
To apply M\"obius inversion, we enlarge the summation range to the family
$\mathcal G$ of spanning subgraphs $G$ of $K_{m,n}$ such that $G-W=H$
and every component meets $W$.
The connected components of each $G\in\mathcal G$ induce a partition
$\pi(G)$ of $W$. Since every component meets $W$, connectedness is
equivalent to $\pi(G)=\hat1$. Therefore, M\"obius inversion yields
\[
 \sum_{G\in\mathcal G,\,\pi(G)=\hat1}w(G)
 =\sum_{\pi\in\Pi(W)}\mu(\pi,\hat1)
       \sum_{G\in\mathcal G,\,\pi(G)\leq\pi}w(G),
 \qquad \mu(\pi,\hat1)=(-1)^{|\pi|-1}(|\pi|-1)!.
\]
For $G\in\mathcal G$, the condition $\pi(G)\leq\pi$ is
equivalent to requiring that each component of $H$ have a nonempty set
of edges to one block of $\pi$ and no edges to any other block, and that
every edge of $G[W]$ have both endpoints in one block. The attachment sets for $U_i$ are counted by
$q_\pi(x_i,y_i)$, while those for $u_j$ and $v_j$ are counted by
$q_\pi(1,2)$ and $q_\pi(2,1)$, respectively. These choices are independent.
Including the factor $3^{-(l+r)}$ from $w(G)$ gives the attachment factor
\[
 \Bigl(\frac{q_\pi(1,2)}3\Bigr)^l
 \Bigl(\frac{q_\pi(2,1)}3\Bigr)^r
 \prod_{i=1}^Nq_\pi(x_i,y_i).
\]
For an allowed edge within $W$, summing over its absence and presence
gives $2$ if it is neither $e$ nor $f$, and
$1+(1-1/\theta)=z/2$ otherwise. A forbidden edge must be absent
and contributes $1$.
Combining these attachment and edge factors with $\mu(\pi,\hat1)$ gives the
coefficients in Table~\ref{tab:partition-products}. For example, $ab|cd$
and $ad|bc$ contribute $-\alpha\beta z^2/4$ and $-4\alpha\beta$ to the
coefficient of $\rho_{2,4}$, respectively.
{Summing the rows gives the expansion of $\Phi$ in
\autoref{lem:four-endpoint-counts}.}
\end{proof}

To choose $\theta$, first suppose that $H$ is connected.
Then $l=r=0$, $N=1$,
$x_1=2^{n-2}$, and $y_1=2^{m-2}$. Hence,
\[
 \Phi=(zx_1y_1-x_1-y_1)^2
       -\frac{x_1y_1+1}{2}(z-2)^2.
\]
The first square vanishes when $z=x_1^{-1}+y_1^{-1}$, which corresponds to
\begin{equation}\label{eq:common-center}
 \theta\coloneqq\frac{1}{2-2^{1-m}-2^{1-n}}.
\end{equation}
We fix this value of $\theta$ for the rest of the section.
Then $1/2<\theta\leq2/3$ and $0<z\leq1$, so the remaining term is
strictly negative. For disconnected $H$, we first bound the component products.

\subsection{Componentwise product estimates}

In the next two lemmas, we use the products in
Table~\ref{tab:partition-products} for $N\geq1$ real pairs $(x_i,y_i)$,
without assuming that they arise from a graph. Also put

\[
 \rho_x\coloneqq\prod_{i=1}^N x_i,\qquad
 \rho_y\coloneqq\prod_{i=1}^N y_i.
\]

For $1\leq i\leq N$, suppose $b_i,d_i\geq c_i\geq0$ and
$a_i+c_i\geq b_i+d_i$. Then
\begin{equation}\label{eq:product-comparison}
 \prod_{i=1}^N a_i+\prod_{i=1}^N c_i
 \geq\prod_{i=1}^N b_i+\prod_{i=1}^N d_i.
\end{equation}
Indeed, $a_i\geq c_i+(b_i-c_i)+(d_i-c_i)$.
The product of these lower bounds expands into
$\prod b_i+\prod d_i-\prod c_i$ plus nonnegative mixed terms.

{Since $Q(\alpha,\beta)$ is affine in each variable, its four
corner values determine it on $[0,1]^2$. The next lemma bounds
$Q(0,0)$, $Q(1,0)$, and $Q(0,1)$ from above.}

\begin{lemma}\label{lem:component-squares}
Let $N\geq1$ be an integer and $x_i,y_i\geq2$ be real numbers for
$1\leq i\leq N$. Then
\[
 \rho_{2,3}\geq \rho_x^2+\rho_y^2-2,\qquad
 \rho_{2,3}+\rho_x^2-2\rho_{3,1}\geq(\rho_y-2^N)^2,\qquad
 \rho_{2,3}+\rho_y^2-2\rho_{3,2}\geq(\rho_x-2^N)^2.
\]
\end{lemma}

\begin{proof}
For the first inequality, iterate the comparison
valid for $x,y,u,v\geq2$:

\[
 (x^2+y^2-2)(u^2+v^2-2)-(x^2u^2+y^2v^2-2)
 =(x^2-2)(v^2-2)+(y^2-2)(u^2-2)-2\geq0.
\]

For the second inequality, we first minimize the left-hand side over
the $x_i$; the third follows by symmetry. The identity
\[
 (x_i^2+y_i^2-2)+x_i^2-2(x_i^2+2y_i-3)=(y_i-2)^2,
\]
together with $x_i^2+2y_i-3\geq x_i^2\geq0$ and
\eqref{eq:product-comparison}, implies
$\rho_{2,3}+\rho_x^2-2\rho_{3,1}\geq0$.
This expression is affine in each $x_j^2$. Its coefficient is the
same nonnegative difference for the other $N-1$ pairs
(zero when $N=1$). Therefore, decreasing each $x_i$ to $2$ gives
\[
 \rho_{2,3}+\rho_x^2-2\rho_{3,1}
 \geq\prod_{i=1}^N(y_i^2+2)+4^N-2\prod_{i=1}^N(2y_i+1).
\]
Apply \eqref{eq:product-comparison} successively with common lower
bounds $y_i^2$ and $2y_i$:
\[
 \prod_{i=1}^N(y_i^2+2)-\prod_{i=1}^Ny_i^2
 \geq2\Bigl(\prod_{i=1}^N(y_i^2+1)-\prod_{i=1}^Ny_i^2\Bigr)
 \geq2\Bigl(\prod_{i=1}^N(2y_i+1)-\prod_{i=1}^N(2y_i)\Bigr).
\]
The second step uses $y_i^2\geq2y_i$.
{Substituting into the lower bound obtained at $x_i=2$ gives
\[
 \rho_{2,3}+\rho_x^2-2\rho_{3,1}
 \geq\rho_y^2+4^N-2^{N+1}\rho_y=(\rho_y-2^N)^2.
\]}
\end{proof}

We now bound the remaining corner value $Q(1,1)$ and the
product $\rho_{3,0}$.

\begin{lemma}\label{lem:component-products}
Under the hypotheses and with the notation of
\autoref{lem:component-squares}, we have
\begin{equation}
 \rho_{3,0}\leq(3/2)^N\rho_x\rho_y,\qquad
 Q(1,1)\leq \rho_x^2+\rho_y^2-2.
 \label{eq:attachment-quadratic-bound}
\end{equation}
\end{lemma}

\begin{proof}
{For the first inequality, note that
\[
 \tfrac32x_iy_i-(x_iy_i+x_i+y_i-3)
 =\tfrac12(x_i-2)(y_i-2)+1\geq0.
\]
Multiplying these inequalities over $i=1,\ldots,N$ gives
$\rho_{3,0}\leq(3/2)^N\rho_x\rho_y$.}

For the second inequality in \eqref{eq:attachment-quadratic-bound}, put $\Delta=\rho_x^2+\rho_y^2-2-Q(1,1)$.
We show that $\Delta$ is nondecreasing in each coordinate and then
evaluate it when all $x_i=y_i=2$.
Applying \eqref{eq:product-comparison} with common lower factors
$2(x_i+y_i-2)$ gives
\[
 \rho_{3,1}+\rho_{3,2}\leq\rho_{2,3}+\rho_4,\qquad
 2\rho_{3,0}\leq\rho_{2,4}+\rho_4.
\]
Also, $\rho_{2,3}\geq\rho_{3,2}$ and
$\rho_{2,4}\geq\rho_{3,0}\geq\rho_4$ factorwise.
{In the derivative formulas below, a superscript $\widehat j$ or
$\widehat{j,k}$ on a product indicates omission of the factors with the
specified indices.
For $1\leq j\leq N$, differentiation gives
\[
 \begin{aligned}
 \frac12\frac{\partial\Delta}{\partial x_j}
 ={}&x_j\bigl(\rho_{2,3}^{\widehat j}+(\rho_x^{\widehat j})^2-2\rho_{3,1}^{\widehat j}\bigr)
       +4y_j(\rho_{2,4}^{\widehat j}-\rho_{3,0}^{\widehat j})-2\rho_{3,2}^{\widehat j}+6\rho_4^{\widehat j}-4\rho_{3,0}^{\widehat j}\\
 \geq{}&2\bigl(\rho_{2,3}^{\widehat j}+(\rho_x^{\widehat j})^2+\rho_{2,4}^{\widehat j}-2\rho_{3,1}^{\widehat j}-\rho_{3,2}^{\widehat j}\bigr)
       +6(\rho_{2,4}^{\widehat j}+\rho_4^{\widehat j}-2\rho_{3,0}^{\widehat j}).
 \end{aligned}
\]
For $N\geq2$, the inequality follows from \autoref{lem:component-squares}
applied to the indices other than $j$, together with
$\rho_{2,4}^{\widehat j}\geq\rho_{3,0}^{\widehat j}$ and $x_j,y_j\geq2$.
For $N=1$, all omitted products equal $1$, so the derivative and both
brackets vanish.
The second bracket is nonnegative by the product comparison above.
For $N\geq2$ and $k\ne j$, half the derivative of the first bracket
with respect to $y_k$ is
\[
 y_k(\rho_{2,3}^{\widehat{j,k}}-\rho_{3,2}^{\widehat{j,k}})
 +x_k\rho_{2,4}^{\widehat{j,k}}-2\rho_{3,1}^{\widehat{j,k}}
 \geq2(\rho_{2,3}^{\widehat{j,k}}+\rho_{2,4}^{\widehat{j,k}}-\rho_{3,1}^{\widehat{j,k}}-\rho_{3,2}^{\widehat{j,k}})\geq0.
\]
Hence, the first bracket is minimized when $y_i=2$ for every $i\ne j$.
At this boundary, $\rho_{2,4}^{\widehat j}\geq\rho_{3,2}^{\widehat j}$
factorwise, since $4x_i-2\geq2x_i+1$.
Applying \autoref{lem:component-squares} to the remaining $N-1$ pairs
proves that the first bracket is nonnegative.
Thus, $\partial\Delta/\partial x_j\geq0$ for every $j$, and symmetry
gives $\partial\Delta/\partial y_j\geq0$ as well.}
Consequently,
\[
 \begin{aligned}
 \Delta
 &\geq8\cdot4^N+5\cdot6^N-12\cdot5^N-2
 =30\bigl(6^{N-1}+4^{N-1}-2\cdot5^{N-1}\bigr)
       +2(4^{N-1}-1)\geq0,
 \end{aligned}
\]
which completes the proof.
\end{proof}

\subsection{Proof of \autoref{thm:main-connected}}

{We return to a spanning subgraph $H$ of $K_{m,n}-W$.
Each of its $N$ components containing an edge meets both vertex classes,
so $x_i,y_i\geq2$. The next bound retains the term involving
$\rho_{2,4}$ for the final strict inequality.

\begin{lemma}\label{lem:extension-remainder}
For $m,n\geq3$ and $\theta$ as in~\eqref{eq:common-center},
every spanning subgraph $H$ of $K_{m,n}-W$ with $N\geq1$ satisfies
\[
 \Phi+z^2\Bigl(1+\frac{\alpha\beta}{4}\rho_{2,4}\Bigr)
 \leq2\alpha\beta\,2^{-l-r}.
\]
\end{lemma}

\begin{proof}
We first obtain a quadratic upper bound in $\rho_x,\rho_y$, then
maximize it on $[2^N,\infty)^2$.
The vertex totals in the two classes imply
\begin{equation}
 2^r\rho_x=2^{n-2},\qquad 2^l\rho_y=2^{m-2},\qquad
 z\rho_x\rho_y=2^{-l}\rho_x+2^{-r}\rho_y.
 \label{eq:vertex-totals}
\end{equation}
Interpolating the four corner bounds in
\autoref{lem:component-squares} and \autoref{lem:component-products} gives
\[
 Q(\alpha,\beta)\leq
 \alpha\rho_x^2+\beta\rho_y^2-2\alpha\beta
 -(1-\alpha)(1-\beta)(\rho_x^2+\rho_y^2-2)
 -\alpha(1-\beta)(\rho_y-2^N)^2-\beta(1-\alpha)(\rho_x-2^N)^2.
\]
To bound the remaining terms in $\Phi$, we apply the following inequalities, which hold for $x_i,y_i\geq2$:
\[
 \rho_1\leq\rho_x^2\rho_y^2-1,\qquad
 \rho_{2,1}\geq\rho_x\rho_y^2,\qquad
 \rho_{2,2}\geq\rho_x^2\rho_y.
\]
The first inequality in \autoref{lem:component-products} controls the
remaining positive term involving $\rho_{3,0}$. Hence,
\begin{equation}
 \Phi+z^2\Bigl(1+\frac{\alpha\beta}{4}\rho_{2,4}\Bigr)
 \leq Q(\alpha,\beta)+(2^{-l}\rho_x+2^{-r}\rho_y)\cdot\bigl[
 (2^{-l}-2\alpha)\rho_x+(2^{-r}-2\beta)\rho_y
 +2\alpha\beta(3/2)^N\bigr].
 \label{eq:centered-extension-bound}
\end{equation}

Substitute the upper bound for $Q(\alpha,\beta)$ into
\eqref{eq:centered-extension-bound}. With $l,r,N$ fixed, regard the
resulting quadratic as a function of independent variables
$\rho_x,\rho_y\geq2^N$. We show that it is nonincreasing in each variable
on this larger domain.
Apply convexity of $v\mapsto v^l$ and $v\mapsto v^r$ at
$2/3=\tfrac13\cdot1+\tfrac23\cdot\tfrac12$ to obtain
$3\alpha\leq1+2^{1-l}$ and $3\beta\leq1+2^{1-r}$.
In particular,
\[
 0\leq\alpha-2^{-l}\leq\frac{1-\alpha}{2},\qquad
 0\leq1-2^{-r}\leq2(1-\beta).
\]
Half the partial derivative of the quadratic polynomial with respect
to $\rho_x$ is
\[
 (2\alpha-1+2^{-2l}-2^{1-l}\alpha)\rho_x
 +(2^{-l-r}-\beta2^{-l}-\alpha2^{-r})\rho_y
 +2^N\beta(1-\alpha)+\alpha\beta(3/2)^N2^{-l}.
\]
The coefficient of $\rho_x$ equals
$-(1-2^{-l})(1-2\alpha+2^{-l})\leq0$ by convexity;
that of $\rho_y$ is nonpositive because $2^{-r}\leq\beta$.
Since $\rho_x,\rho_y\geq2^N$ and $(3/2)^N\leq2^N$,
half the derivative is at most
\[
 2^N\Bigl[(1-2^{-l})\bigl((\alpha-2^{-l})(1-2^{-r})
 -(1-\alpha)(1-\beta)\bigr)
 -(\alpha-2^{-l})2^{-l-r}\Bigr]\leq0.
\]
The difference in parentheses is nonpositive by the two bounds from convexity
above, and the remaining term is nonpositive.
The derivative with respect to $\rho_y$ is nonpositive by symmetry.
Thus, the maximum occurs at $(2^N,2^N)$.

At $(\rho_x,\rho_y)=(2^N,2^N)$, the upper bound for $Q(\alpha,\beta)$ is
\begin{equation}
 4^N\bigl[\alpha+\beta-2(1-\alpha)(1-\beta)\bigr]
 +2(1-\alpha)(1-\beta)-2\alpha\beta.
 \label{eq:upper-bound-for-K}
\end{equation}
Put $\eta\coloneqq\alpha+\beta-2^{-l}-2^{-r}$.
To combine the remaining boundary term with
\eqref{eq:upper-bound-for-K}, use $2\cdot3^N\leq4^N+2$, valid for
$N\geq1$ by induction:
\begin{align}
 &4^N\bigl[\eta^2-(\alpha+\beta)^2\bigr]
       +2\alpha\beta\,3^N(\alpha+\beta-\eta)\notag\\
 &\leq4^N\bigl[\eta(\eta-\alpha\beta)
       -(\alpha+\beta)(\alpha+\beta-\alpha\beta)\bigr]
       +2\alpha\beta(2^{-l}+2^{-r}).
 \label{eq:remaining-term-estimate}
\end{align}

Since $\alpha^2\leq2^{-l}\leq\alpha$, and similarly for $\beta$,
\[
 0\leq\eta\leq\alpha(1-\alpha)+\beta(1-\beta)
 \leq\alpha\beta+(1-\alpha)(1-\beta).
\]
The last inequality is equivalent to $(\alpha+\beta-1)^2\geq0$.
The convexity bounds above also give
$\eta\leq(2-\alpha-\beta)/2$. Consequently, the coefficient of $4^N$
obtained by adding the bounds in~\eqref{eq:upper-bound-for-K} and
\eqref{eq:remaining-term-estimate} satisfies
\begin{equation}
 \eta(\eta-\alpha\beta)
 -(1-\alpha)(1-\beta)(2-\alpha-\beta)
 \leq-\tfrac12(1-\alpha)(1-\beta)(2-\alpha-\beta).
 \label{eq:isolated-vertex-estimate}
\end{equation}

Add the bounds in \eqref{eq:upper-bound-for-K} and
\eqref{eq:remaining-term-estimate}, apply
\eqref{eq:isolated-vertex-estimate}, and use $4^N\geq4$:
\[
 \begin{aligned}
 \Phi+z^2\Bigl(1+\frac{\alpha\beta}{4}\rho_{2,4}\Bigr)
 &\leq2(\alpha+\beta-1)(1-\alpha)(1-\beta)
       -2\alpha\beta(1-2^{-l}-2^{-r})\\
 &\leq2\alpha\beta\,2^{-l-r}
       -2(1-\alpha)^2(1-\beta)^2
 \leq2\alpha\beta\,2^{-l-r}.
 \end{aligned}
\]
For the second inequality, we used
$1-2^{-l}-2^{-r}\geq(1-\alpha)(1-\beta)-2^{-l-r}$,
which follows from $2^{-l}\leq\alpha$ and $2^{-r}\leq\beta$.
\end{proof}
}

{We use a lower bound for $\rho_{2,4}$ to absorb the positive
remainder in \autoref{lem:extension-remainder}, and treat the edgeless
case directly.}

\begin{lemma}\label{lem:centered-extensions}
For $m,n\geq3$ and $\theta$ as in~\eqref{eq:common-center},
every spanning subgraph $H$ of $K_{m,n}-W$ satisfies
\[
 j_{ef}^H-2\theta j_e^H+\theta^2j^H<0.
\]
\end{lemma}

{\begin{proof}
Since $\theta^2 3^{l+r}>0$, it suffices to prove $\Phi<0$.
If $N=0$, then $l=m-2\geq1$ and $r=n-2\geq1$.
In this case $z=2^{-l}+2^{-r}$.
Taking $N=0$ in \autoref{lem:four-endpoint-counts} and completing the square gives
\[
 \Phi=-(1-z)\bigl[(\alpha+\beta-1)^2
         +z-\alpha^2-\beta^2\bigr]
       -\frac{z^2}{4}\alpha\beta<0.
\]
Indeed, $z\leq1$, while
$\alpha^2=(4/9)^l\leq2^{-l}$ and
$\beta^2=(4/9)^r\leq2^{-r}$; the last term is strictly negative.

Now suppose that $N\geq1$. We have $\rho_{2,4}\geq2\rho_x\rho_y-2$.
For $N=1$, equality holds. For $N\geq2$, the factorwise bound
$2(x_iy_i-1)\geq\tfrac32x_iy_i$ implies
$\rho_{2,4}\geq(3/2)^N\rho_x\rho_y\geq2\rho_x\rho_y$.
By \eqref{eq:vertex-totals} and the arithmetic--geometric mean inequality,
we obtain
\[
 z^2\rho_x\rho_y
 =\frac{(2^{-l}\rho_x+2^{-r}\rho_y)^2}{\rho_x\rho_y}
 \geq4\cdot2^{-l-r}.
\]
Therefore, \autoref{lem:extension-remainder} yields
\[
 \Phi\leq\frac{\alpha\beta}{2}z^2\rho_x\rho_y
 -z^2\Bigl(1+\frac{\alpha\beta}{4}\rho_{2,4}\Bigr)
 \leq-z^2(1-\alpha\beta/2)<0.\qedhere
\]
\end{proof}
}

\begin{proof}[Proof of \autoref{thm:main-connected}]
By \autoref{cor:pair-order}, it suffices to consider disjoint edges
$e=ab$ and $f=cd$.
When $m=2$, a spanning subgraph is connected precisely when every
vertex of $R$ has a nonempty neighbourhood and at least one has both
vertices of $L$ as neighbours. Counting the three possible nonempty
neighbourhoods, with the prescribed edges when required, gives
\[
 j(K_{2,n})=3^n-2^n,
 j_e(K_{2,n})=j_f(K_{2,n})=2\cdot3^{n-1}-2^{n-1},
 j_{ef}(K_{2,n})=4\cdot3^{n-2}-2^{n-2}.
\]
Consequently,
$j_e(K_{2,n})j_f(K_{2,n})-j(K_{2,n})j_{ef}(K_{2,n})
=6^{n-2}>0$.
The case $n=2$ follows by symmetry. Now let $m,n\geq3$ and choose $\theta$ as in \eqref{eq:common-center}.
Sum the inequalities in \autoref{lem:centered-extensions}
over all spanning subgraphs $H$ of $K_{m,n}-W$:
\[
 j_{ef}(K_{m,n})-2\theta j_e(K_{m,n})
       +\theta^2j(K_{m,n})
 =\sum_H\bigl(j_{ef}^H-2\theta j_e^H+\theta^2j^H\bigr)<0.
\]
Since $j(K_{m,n})>0$, the desired strict inequality now
follows from \eqref{eq:centering-identity}.
\end{proof}

{\autoref{lem:centered-extensions} gives strict negative
correlation after conditioning on a fixed restriction. 

\begin{corollary}\label{cor:conditional-connected}
For $m,n\geq2$, let $H$ be a spanning subgraph of $K_{m,n}-W$,
where $W=\{a,b,c,d\}$ is as above.
For any distinct edges $e',f'\in\{ab,ad,cb,cd\}$, the corresponding
extension counts satisfy
\[
 j_{e'}^Hj_{f'}^H-j_{e'f'}^Hj^H>0.
\]
\end{corollary}

\begin{proof}
Swapping $a,c$ or $b,d$ preserves $H$ and makes the four one-edge
inclusion counts equal. Under these swaps, every pair $\{e',f'\}$ is
equivalent to $\{ab,ad\}$, $\{ab,cb\}$, or $\{ab,cd\}$.
Set $t=1$ in \eqref{eq:connected-pair-comparison} and restrict
the sum to spanning subgraphs $K$ of $K_{m,n}-\{b,d\}$ satisfying
$K-\{a,c\}=H$.
This set is invariant under interchanging $a,c$, and the summands
are nonnegative by the proof of
\autoref{prop:connected-pair-comparison}.
Hence $j_{ab,cd}^H\geq j_{ab,ad}^H$. Exchanging the parts gives $j_{ab,cd}^H\geq j_{ab,cb}^H$.
Thus, 
\[
 j_{e'}^Hj_{f'}^H-j_{e'f'}^Hj^H
 \geq j_{ab}^Hj_{cd}^H-j^Hj_{ab,cd}^H.
\]
It remains to prove strict positivity for the disjoint pair
$e=ab$ and $f=cd$.

Adding to $H$ all vertices of $W$ and all edges incident with $W$ gives a connected extension, so $j^H>0$. If $m=2$ or $n=2$, then $K_{m,n}-W$ is edgeless, so the claim follows from
\autoref{thm:main-connected}. Now suppose $m,n\geq3$ and choose $\theta$ as in
\eqref{eq:common-center}. Since $j_e^H=j_f^H$,
\autoref{lem:centered-extensions} gives
\[
 j_e^Hj_f^H-j^Hj_{ef}^H
 =(j_e^H-\theta j^H)^2
   -j^H(j_{ef}^H-2\theta j_e^H+\theta^2j^H)>0.\qedhere
\]
\end{proof}
}

\section*{Acknowledgements}
This work is supported by the Guangdong Basic and Applied Basic Research Foundation (Grant~No.~2026A1515012237).

\section*{Declaration of generative AI and AI-assisted technologies}
The authors supplied the original manuscript and directed the research
and revision. OpenAI Codex assisted in deriving the componentwise product
estimates in \autoref{lem:component-products} and developing their use
in the proof of \autoref{lem:extension-remainder}. This work consolidated
previously separate estimates within the existing proof strategy
using a common parameter. Codex assisted with literature searches, exploratory
computations, and language editing. The authors take full responsibility for the
manuscript.

\end{document}